\documentclass[a4paper]{article}
\usepackage{mathptmx}
\usepackage{amscd,amssymb,amsmath,amsthm}

\makeatletter

\renewcommand{\@seccntformat}[1]
{{\csname the#1\endcsname}.\hspace{0.3em}}

\renewcommand{\section}{\@startsection
{section}
{1}
{0mm}
{-1.5\baselineskip}
{\baselineskip}
{\bfseries\normalsize}}

\renewcommand{\subsection}{\@startsection
{subsection}
{2}
{0mm}
{-\baselineskip}
{0.5\baselineskip}
{\normalsize\itshape}}

\renewcommand{\subsubsection}{\@startsection
{subsubsection}
{3}
{0mm}
{-.5\baselineskip}
{-2mm}
{\normalsize\itshape}}

\makeatother

\theoremstyle{plain}
\newtheorem*{theorem*}{Theorem}
\newtheorem{theorem}{Theorem}[section]
\newtheorem{lemma}{Lemma}[section]
\newtheorem{corollary}[lemma]{Corollary}
\newtheorem{prop}[lemma]{Proposition}

\newtheorem*{corollary*}{Corollary}
\newtheorem*{CNT}{Courant's nodal domain theorem}
\newtheorem*{CST}{Cheng's structure theorem}

\theoremstyle{definition}
\newtheorem*{defin*}{Definition}
\newtheorem{defin}{Definition}[section]

\theoremstyle{remark}
\newtheorem*{remark*}{Remark}

\newtheorem*{quest*}{Question}

\DeclareMathAlphabet{\matheur}{U}{eur}{m}{n}
\DeclareMathAlphabet{\matheus}{U}{eus}{m}{n}
\DeclareMathAlphabet{\matheuf}{U}{euf}{m}{n}

\numberwithin{equation}{section}

\newcommand{\abs}[1]{\left\lvert#1\right\rvert}

\DeclareMathOperator{\dist}{dist}

\DeclareMathOperator{\ord}{ord}

\DeclareMathOperator{\Int}{Int}

\begin{document}

\author{Gerasim  Kokarev
\\ {\small\it Mathematisches Institut der Universit\"at M\"unchen }
\\ {\small\it Theresienstr. 39, D-80333 M\"unchen, Germany}
\\ {\small\it Email: {\tt Gerasim.Kokarev@mathematik.uni-muenchen.de}}
}

\title{On multiplicity bounds for Schrodinger eigenvalues on Riemannian surfaces}
\date{}
\maketitle

\begin{abstract}
\noindent
A classical result by Cheng in 1976, improved later by Besson and Nadirashvili, says that the multiplicities of the eigenvalues of the Schrodinger operator $(-\Delta_g+\nu)$, where $\nu$ is $C^\infty$-smooth, on a compact Riemannian surface $M$ are bounded in terms of the eigenvalue index and the genus of $M$. We prove that these multiplicity bounds hold for an $L^p$-potential $\nu$, where $p>1$. We also discuss similar multiplicity bounds for Laplace eigenvalues on singular Riemannian surfaces.
\end{abstract}

\medskip
\noindent
{\small

}

%
%
%


\section{Introduction and statements of results}
\label{intro}
\subsection{Multiplicity bounds}
Let $M$ be a connected compact surface. For a Riemannian metric $g$ and a $C^\infty$-smooth function $\nu$ on $M$ we denote by
$$
\lambda_0(g,\nu)<\lambda_1(g,\nu)\leqslant\ldots\lambda_k(g,\nu)\leqslant\ldots
$$
the eigenvalues of the Schrodinger operator $(-\Delta_g+\nu)$. If $M$ has a non-empty boundary, we assume that the Dirichlet boundary condition is imposed.

The following theorem is an improved version of the statement originally discovered by Cheng~\cite{Cheng} in 1976. For closed orientable surfaces it is due to Besson~\cite{Be}, and for general closed surfaces due to Nadirashvili~\cite{Na}; multiplicity bounds for general boundary value problems have been obtained in~\cite{KKP}.
\begin{theorem}
\label{t0}
Let $(M,g)$ be a smooth compact surface, possibly with boundary. Then for any $C^\infty$-smooth function $\nu$ on $M$ the multiplicity $m_k(g,\nu)$ of an eigenvalue $\lambda_k(g,\nu)$ satisfies the inequality
$$
m_k(g,\nu)\leqslant 2(2-\chi-l)+2k+1, \qquad k=1,2,\ldots,
$$
where $\chi$ stands for the Euler-Poincare number of $M$ and $l$ is the number of boundary components.
\end{theorem}
Above we assume that $l=0$ for closed surfaces. Mention that even the fact that eigenvalue multiplicities on Riemannian surfaces are bounded is by no means trivial, and as is known~\cite{CDV86,CDV87}, fails in higher dimensions, unless some specific hypotheses on a Riemannian metric or a potential are imposed. The key ingredient in the proof of Theorem~\ref{t0} is the so-called {\em Cheng's structure theorem}~\cite{Cheng}: for any solution $u$ to the Schrodinger equation with a smooth potential and any interior point $p\in M$ there exists a neighbourhood of $p$ and its diffeomorphism onto a ball in $\mathbf{R}^2$ centred at the origin that maps the nodal set of $u$ onto the nodal set of a homogeneous harmonic polynomial. This statement is based on a local approximation of solutions by harmonic homogeneous polynomials~\cite{Bers}, and in particular, implies that the nodal set of a solution $u$ is locally homeomorphic to its tangent cone. The latter property of nodal sets does not hold in higher dimensions, see~\cite{BM}. 
The structure theorem holds for sufficiently smooth solutions to the Schrodinger equation, see Appendix~\ref{stru}, and consequently, the multiplicity bounds in Theorem~\ref{t0} hold for Holder continuous potentials. Based on Cheng's structure theorem, the multiplicity bounds for various eigenvalues problems have been extensively studied in the literature. We refer to the papers~\cite{CDV87,HoNa2,HoNa,KKP} and references there for the details.

The purpose of this paper is to show that the multiplicity bounds continue to hold for rather weak potentials when no similar structure theorem for nodal sets is available. For a given real number $\delta\in(0,2)$ we consider the class $K^{2,\delta}(M)$, introduced in~\cite{AS,Si}, formed by absolutely integrable potentials $\nu$ such that
\begin{equation}
\label{pot}
\sup_{x\in M}\int\limits_{B(x,r)}\abs{{x-y}}^{-\delta}\abs{\nu(y)}dVol_g(y)\to 0\qquad\text{as}\quad r\to 0,
\end{equation}
where the absolute value $\abs{x-y}$ above denotes the distance between $x$ and $y$ in the background metric $g$.  It is a straightforward consequence of the Holder inequality that any $L^p$-integrable function with $p>1$ belongs to $K^{2,\delta}$ for some positive $\delta$. However, unlike the traditional $L^p$-hypothesis the potentials from $K^{2,\delta}(M)$ include certain physically important cases, see~\cite{AS,Si}.

The hypothesis that $\nu\in K^{2,\delta}(M)$ implies that the measures $d\mu^\pm=\nu^\pm\mathit{dVol}_g$, where $\nu^+$ and $\nu^-$ are positive and negative parts of $\nu$ respectively,  are $\delta$-uniform:
$$
\mu^\pm(B(x,r))\leqslant Cr^\delta,\qquad\text{for any}\quad r>0\text{ ~and~ }x\in M, 
$$
and some constant $C$. By the results of Maz'ja~\cite{Ma}, see also~\cite{GK}, for such measures $\mu^\pm$ the Sobolev space $W^{1,2}(M,\mathit{Vol}_g)$ embeds compactly into  $L_2(M,\mu^\pm)$. By standard perturbation theory~\cite{Kato}, see also~\cite{Ma,Si}, we then conclude that the spectrum of the Schrodinger operator $(-\Delta_g+\nu)$ is discrete, bounded from below, and all eigenvalues have finite multiplicities. Our main result says that they satisfy the same multiplicity bounds.
\begin{theorem}
\label{t1}
Let $(M,g)$ be a smooth compact surface, possibly with boundary. Then for any absolutely integrable potential $\nu$ from $K^{2,\delta}(M)$, where $\delta\in (0,2)$, the multiplicity $m_k(g,\nu)$ of an eigenvalue $\lambda_k(g,\nu)$ satisfies the inequality
$$
m_k(g,\nu)\leqslant 2(2-\chi-l)+2k+1, \qquad k=1,2,\ldots,
$$
where $\chi$ stands for the Euler-Poincare number of $M$ and $l$ is the number of boundary components.
\end{theorem}
For the first eigenvalue $\lambda_1(g,\nu)$ the above multiplicity bound is sharp when $M$ is homeomorphic to a sphere $\mathbb S^2$ or a projective plane $\mathbb{R}P^2$. When a potential $\nu$ is smooth, there is an extensive literature, see~\cite{CDV87,Na,Se} and reference there, devoted to sharper multiplicity bounds for the first eigenvalue. In addition,  in~\cite{HoNa2,HoNa} the authors show that when $M$ is a sphere or a disk the multiplicity bounds in Theorem~\ref{t0} can be improved to $m_k(g,\nu)\leqslant 2k-1$ for $k\geqslant 2$.  
We have made no effort in improving our results in these directions. However, it is worth mentioning that the main topological result in~\cite{Se} does yield a sharper multiplicity bound for $\lambda_1(g,\nu)$ for some closed surfaces when a potential $v$ belongs to the space $K^{2,\delta}(M)$. More precisely, if $M$ is a closed surface whose Euler-Poincare number is negative, $\chi<0$, then \cite[Theorem 5]{Se} implies that $m_1(g,\nu)\leqslant 5-\chi$ for any potential $v\in K^{2,\delta}(M)$. By the results in~\cite{CDV87} this bound is sharp for $\mathbb{T}^2\#\mathbb{T}^2$ and $\#n\mathbb{R}P^2$, where $n=3,4,5$.

The multiplicity bounds in Theorem~\ref{t0} also hold for eigenvalue problems on singular Riemannian surfaces; we discuss them in detail in Sect.~\ref{sing}. The proof of Theorem~\ref{t1} is based on the delicate study of the nodal sets of Schrodinger eigenfunctions that we describe below. 

\subsection{Nodal sets of eigenfunctions}
Let $u$ be a solution to the eigenvalue problem
\begin{equation}
\label{sch}
(-\Delta_g +\nu)u=\lambda u\qquad\text{on~ }M,
\end{equation}
where $\nu\in K^{2,\delta}(M)$, and if $\partial M\ne\varnothing$, the Dirichlet boundary hypothesis is assumed. Recall that by results in~\cite{Si} such an eigenfunction $u$ is Holder continuous. By $\mathcal N(u)$ we denote its nodal set, that is the set $u^{-1}(0)$. 

By the results in~\cite{HoHo,HoHoNa} combined with the strong unique continuation property~\cite{Saw,SS}, in appropriate local coordinates around an interior point $x_0\in M$ a non-trivial solution $u$ has the form
$$
u(x)=P_N(x-x_0)+O(\abs{x-x_0}^{N+\delta}),\qquad\text{where}\quad x\in U,
$$
where $P_N$ is a homogeneous harmonic polynomial on the Euclidean plane. We refer to Sect.~\ref{prem} for a precise statement. The degree of this approximating homogeneous harmonic polynomial defines the so-called {\em vanishing order} $\ord_x(u)$ for any interior point $x\in M$. Each point $x\in\mathcal N(u)$ has vanishing order at least one, and we define $\mathcal N^2(u)$ as the set of points $x$ whose vanishing order $\ord_x(u)$ is at least two.

The proof of Theorem~\ref{t1} is based on the following key result.
\begin{theorem}
\label{t2}
Let $(M,g)$ be a compact Riemannian surface, possibly with boundary, and $u$ be a non-trivial eigenfunction for the Schrodinger eigenvalue problem~\eqref{sch} with $\nu\in K^{2,\delta}(M)$, where $\delta\in (0,2)$. Then the set $\mathcal N^2(u)$ is finite, and the complement $\mathcal N(u)\backslash\mathcal N^2(u)$ has finitely many connected components. Moreover, for any $x\in\mathcal N^2(u)$ the number of connected components of $\mathcal N(u)\backslash\mathcal N^2(u)$ incident to $x$ is an even integer that is at least $2\ord_x(u)$.
\end{theorem}
The theorem says that the nodal set $\mathcal N(u)$ can be viewed as a graph: the vertices are points from $\mathcal N^2(u)$, and the edges are connected components of $\mathcal N(u)\backslash\mathcal N^2(u)$. This graph structure assigns to each $x\in\mathcal N^2(u)$ its degree $\deg(x)$, that is, the number of edges incident to $x$. If there is an edge that starts and ends at the same point, then it counts twice. The last statement of Theorem~\ref{t2} says that $\deg(x)\geqslant 2\ord_x(u)$ for any $x\in\mathcal N^2(u)$. When the potential $\nu$ is smooth, Theorem~\ref{t2} is a direct consequence of Cheng's structure theorem, and in this case, the degree $\deg(x)$ is precisely $2\ord_x(u)$.

The proof of Theorem~\ref{t2} uses essentially Courant's nodal domain theorem, and is based on topological arguments, which are in turn built on the results in~\cite{HoHo,HoHoNa}. More precisely, one of the key ingredients is the description of prime ends of nodal domains, which leads to a construction of neighbourhoods of $x\in\mathcal N(u)$ where a solution has also a finite number of nodal domains. Our method uses the properties of solutions in the interior of $M$ only; it largely disregards their behaviour at the boundary. Consequently, the main results (Theorems~\ref{t1} and~\ref{t2}) hold for rather general boundary value problems as long as Courant's nodal domain theorem holds, cf.~\cite[Sect.~6]{KKP}. The statement of Theorem~\ref{t2} continues to hold for general solutions to the Schrodinger equation $(-\Delta+V)u=0$ that have  a finite number of nodal domains. Without the latter hypothesis for arbitrary $L^p$-potentials it is unknown even whether the Hausdorff dimension of $\mathcal N^2(u)$ equals zero or not.

The paper is organised in the following way. In Sect.~\ref{prem} we collect the background material on the strong unique continuation property, regularity of nodal sets, and recall the approximation results from~\cite{HoHo,HoHoNa}. Here we also derive a number of consequences of these results that describe qualitative properties of nodal sets; they  are used often in our sequel arguments. In the next section we recall the notion of Caratheodory's prime end and show that prime ends of nodal domains have the simplest possible structure: their impression always consists of a single point. In Sect.~\ref{proofs} we prove Theorems~\ref{t1} and~\ref{t2}. In the last section we discuss multiplicity bounds for eigenvalue problems on surfaces with measures. We show that Laplace eigenvalue problems on singular Riemannian surfaces, such as Alexandrov surfaces of bounded integral curvature, can be viewed as particular instances of such problems. The paper also has an appendix where we give details on Cheng's structure theorem for reader's convenience.

\smallskip
\noindent
{\em Acknowledgements.} Some of our arguments at the end of Sect.~\ref{proofs} (the proof of Lemma~\ref{iso}) are similar in the spirit to the ones in~\cite{KKP}, and I am grateful to Mikhail Karpukhin and Iosif Polterovich for a number of discussions on the related topics. I am also grateful to Yuri Burago for a number of comments on Alexandrov surfaces.

\section{Preliminaries}
\label{prem}
\subsection{Background material}
We start with collecting background material on solutions of the Schrodinger equation, which is used throughout the paper. From now on we assume that a potential $V$ belongs to the space $K^{2,\delta}(M)$, where $\delta\in (0,1)$. The superscript $2$ in the notation for this function space refers to the dimension of $M$. Mention that the space $K^{2,\delta}(M)$ is contained in the so-called {\em Kato space} formed by absolutely integrable functions $V$ such that
$$
\sup_{x\in M}\int\limits_{B(x,r)}\ln(1/\abs{{x-y}})\abs{\nu(y)}dVol_g(y)\to 0\qquad\text{as}\quad r\to 0,
$$
see~\cite{Si}. Consider the Schrodinger equation
\begin{equation}
\label{schro}
(-\Delta_g+V)u=0\qquad\text{on~ }M,
\end{equation} 
understood in the distributional sense. As was mentioned above, by the results in~\cite{Si} its solutions are Holder continuous. They also enjoy the following {\em strong unique continuation property}.
\begin{prop}
\label{uc}
Let $(M,g)$ be a smooth connected compact Riemannian surface, possibly with boundary, and $x_0\in M$ be an interior point. Let $u$ be a non-trivial solution of the Schrodinger equation~\eqref{schro} with $V\in K^{2,\delta}(M)$, where $0<\delta<1$, such that 
$$
u(x)=O(\abs{x-x_0}^\ell)\qquad\text{for any }\ell>0.
$$
Then $u$ vanishes identically on $M$. 
\end{prop}
Prop.~\ref{uc} is a consequence of the results in~\cite{Saw}, where the author proves that a solution $u$ of the Schrodinger equation with the potential $V$ from the Kato space $K^2(M)$ satisfies the unique continuation property: if $u$ vanishes on a non-empty open subset, then it vanishes identically. As was pointed out in~\cite{HoHoNa,SS}, the argument in~\cite{Saw} actually yields the strong unique continuation property.

The following fundamental statement is a combination of the main result in~\cite{HoHo} together with Prop.~\ref{uc}. 
\begin{prop}
\label{vo}
Let $(M,g)$ be a smooth compact Riemannian surface, possibly with boundary, and $u$ be a non-trivial solution of the Schrodinger equation~\eqref{schro} with $V\in K^{2,\delta}(M)$, where $0<\delta<1$. Then for any interior point $x_0\in M$ there exist its coordinate chart $U$ and a non-trivial homogeneous harmonic polynomial $P_N$ of degree $N\geqslant 0$ on the Euclidean plane such that
$$
u(x)=P_N(x-x_0)+O(\abs{x-x_0}^{N+\delta'}),\qquad\text{where}\quad x\in U,
$$
for any $0<\delta'<\delta$.
\end{prop} 
The proposition says that for any point $x\in M$ there is a well-defined {\em vanishing order} $\ord_x(u)$ of a solution $u$ at $x$, understood as the degree of the harmonic polynomial $P_N$. For a positive integer $\ell$ we define the set
$$
\mathcal N^\ell(u)=\{x\in\Int M~|~ \ord_x(u)\geqslant\ell\}.
$$
Clearly, the nodal set $\mathcal N(u)=u^{-1}(0)$ is precisely the set $\mathcal N^1(u)$. Recall that a connected component of $M\backslash\mathcal N(u)$ is called the {\em nodal domain} of $u$. The combination of the Harnack inequality in~\cite{AS,Si} and the unique continuation property implies that a non-trivial solution $u$ has different signs on adjacent nodal domains. Besides, every point $x\in\mathcal N(u)$ belongs to the closure of at least two nodal domains.

Now suppose that $u$ is an eigenfunction, that is, a solution to eigenvalue problem~\eqref{sch}.  The following version of a classical statement is used in sequel.
\begin{CNT}
Let $(M,g)$ be a smooth compact Riemannian surface, possibly with boundary, and $\nu\in K^{2,\delta}(M)$, where $0<\delta<1$. Then each non-trivial eigenfunction $u$ corresponding to the eigenvalue $\lambda_k(g,\nu)$ of eigenvalue problem~\eqref{sch} has at most $(k+1)$ nodal domains.
\end{CNT}
The proof follows standard arguments, see~\cite{CH}. It uses variational characterisation of eigenvalues $\lambda_k(g,\nu)$, the unique continuation property, Prop.~\ref{uc}, and the continuity of eigenfunctions up to the boundary. The latter can be deduced, for example, from the interior regularity~\cite{Si} by straightening the boundary locally and reflecting across it in an appropriate way.

\subsection{Qualitative properties of nodal sets}
Let $u$ be a solution of the Schrodinger equation~\eqref{schro}. If $u$ is $C^1$-smooth, then the implicit function theorem implies that the complement
\begin{equation}
\label{arcs}
\mathcal N^1(u)\backslash\mathcal N^2(u)
\end{equation} 
is a collection of $C^1$-smooth arcs. The following celebrated nodal set regularity theorem due to~\cite{HoHoNa} says that the latter holds under rather weak assumptions on a potential, when a solution $u$ is not necessarily $C^1$-smooth.
\begin{prop}
\label{reg}
Let $u$ be a non-trivial solution of the Schrodinger equation~\eqref{schro} with $V\in K^{2,\delta}(M)$, where $0<\delta<1$. Then any point $x$ in the complement~\eqref{arcs} has a neighbourhood $U\subset M$  such that  the set $\mathcal N^1(u)\cap U$ is the graph of a  $C^{1,\delta}$-smooth function with non-vanishing gradient. Further, if a potential $V$ is $C^{k,\alpha}$-smooth, then such a point $x$ has a neighborhood $U$ such that $\mathcal N^1(u)\cap U$ is the graph of a  $C^{k+3,\alpha}$-smooth function with non-vanishing gradient.
\end{prop}
Below by {\em nodal edges} we call the connected components of $\mathcal N^1(u)\backslash\mathcal N^2(u)$. By Prop.~\ref{reg} they are diffeomorphic to intervals of the real line, and their ends belong to the set $\mathcal N^2(u)$. We say that a nodal edge is {\em incident} to $x\in\mathcal N^2(u)$, if its closure contains $x$. A nodal edge is called the {\em nodal loop}, if  it is incident to one point $x\in\mathcal N^2(u)$ only. In other words, such a nodal edge starts and ends at the same point $x$.

The important consequence of Prop.~\ref{reg} is the statement that nodal edges can not accumulate to another nodal edge. We use this fact to describe a nodal set structure around an isolated point $x\in\mathcal N^2(u)$. 
\begin{corollary}
\label{degree}
Let $(M,g)$ be a smooth compact Riemannian surface, possibly with boundary, and $u$ be a non-trivial solution of the Schrodinger equation~\eqref{schro} with  $V\in K^{2,\delta}(M)$, where $0<\delta<1$. Let $x\in\mathcal N^2(u)$ be an isolated point in $\mathcal N^2(u)$. Then the number of nodal edges incident to $x$ that are not nodal loops is finite. Moreover, any sequence of nodal loops incident to $x$ has to contract to $x$.
\end{corollary}
\begin{proof}
Let $B$ be a neighbourhood of $x$ whose closure does not contain any points in $\mathcal N^2(u)$. We view $B$ as a unit ball in $\mathbf{R}^2$ centered at the origin $x=0$. Suppose that there is an infinite number of nodal edges incident to $x$ that are not nodal loops. Denote by $\Gamma_i$ the connected components of the intersections of these nodal edges with the ball $B$ whose closures $\bar\Gamma_i$ contain $x$. By Prop.~\ref{reg}, each $\bar\Gamma_i$ consist of a piece of a $C^1$-smooth nodal arc and the origin $x$. They form a sequence of compact subsets of $\bar B$, and hence, contain a subsequence that converges to a compact subset $\bar\Gamma_0\subset\bar B$ in the Hausdorff distance. Clearly, the subset $\bar\Gamma_0$  belongs to the nodal set $\mathcal N(u)$ and contains the origin $x=0$. 
Since the subsets $\bar\Gamma_i$ contain points on the boundary $\partial B$, then so does $\bar\Gamma_0$; in particular, the limit subset $\bar\Gamma_0$ does not coincide with $x$. Since the origin $x$ is the only higher order nodal point in $\bar B$, then $\bar\Gamma_0\backslash\{x\}$ is the union of pieces of $C^1$-smooth nodal edges. Without loss of generality, we may assume that the sequence $\bar\Gamma_i$ converges to a subset $\bar\Gamma_0$ such that $\bar\Gamma_0\backslash\{x\}$ is a piece of a nodal edge. Now to get a contradiction we may either appeal to Prop.~\ref{reg} directly, or argue in the following fashion. Let $x_i\in\bar\Gamma_i\cap\partial B$ be a sequence of points that converges to a point $x_0\in\bar\Gamma_0\cap\partial B$. We consider the two cases.

\noindent
{\em Case~1: the complement $\bar\Gamma_0\backslash\{x\}$ belongs to a nodal edge that intersects $\partial B$ at $x_0$ transversally.} By Prop.~\ref{vo}, it is straightforward to see that the tangent line to $\Gamma_0$ at $x_0$ is precisely the kernel of an approximating linear function $P_1$ at $x_0$. Since $\Gamma_0$ intersects $\partial B$ at $x_0$ transversally, we conclude that the sequence $P_1((x_i-x_0)/\abs{x_i-x_0})$ is bounded away from zero for all sufficiently large $i$. On the other hand, by Prop.~\ref{vo} we obtain $P_1(x_i-x_0)=O(\abs{x_i-x_0}^{1+\delta})$, and arrive at a contradiction.

\noindent
{\em Case~2: the complement $\bar\Gamma_0\backslash\{x\}$ belongs to a nodal edge that is tangent to $\partial B$ at $x_0$.} Then there exists a sufficiently small ball $B_0$ centred at $x_0$ such that $\Gamma_0$ intersects $\partial B_0$ transversally. Choosing a sequence of points $x'_i\in\Gamma_i\cap\partial B_0$ that converges to a point $x'_0\in\Gamma_0\cap B_0$, and arguing in the fashion similar to the one in Case~1, we again arrive at a contradiction.

Now we demonstrate the last statement of the lemma. Suppose that there is a sequence of nodal loops incident to $x$ that do not contract to $x$. Choosing a subsequence and a sufficiently small neighbourhood $B$ of $x$, we may assume that each nodal loop intersects with $\partial B$. Then the argument above shows that this sequence has to be finite.
\end{proof}

We proceed with another statement on local properties of the nodal set near an isolated point $x\in\mathcal N^2(u)$.
\begin{corollary}
\label{degree2}
Let $(M,g)$ be a smooth compact Riemannian surface, possibly with boundary, and $u$ be a non-trivial solution of the Schrodinger equation~\eqref{schro} with  $V\in K^{2,\delta}(M)$, where $0<\delta<1$. Let $x\in\mathcal N^2(u)$ be an isolated point in $\mathcal N^2(u)$. Then there exists a neigbourhood $B$ of $x$, viewed as a ball in the Euclidean plane, such that the zeroes of $u$ on $\partial B$ are precisely the intersections of the connected components of $\mathcal N^1(u)\backslash\mathcal N^2(u)$ incident to $x$ with $\partial B$.
\end{corollary}
\begin{proof}
First, since $x$ is isolated in $\mathcal N^2(u)$, one can choose a neighbourhood $B$ such that it does not contain other points from $\mathcal N^2(u)$. Thus, for a proof of the lemma it is sufficient to show that the point $x$ is not a limit point of the nodal edges that are not incident to $x$. This can be demonstrated following an argument similar to the one used in the proof of Corollary~\ref{degree}.
\end{proof}

Let $x\in\mathcal N^2(u)$ be a point isolated in $\mathcal N^2(u)$ such that the number of nodal edges incident to $x$ is finite. The number of these nodal edges, where nodal loops are counted twice, is a characteristic of a point $x$, called the {\em degree} $\deg(x)$. It is closely related to the vanishing order $\ord_x(u)$. More precisely, if a solution $u$ is sufficiently smooth, then by Cheng's structure theorem~\cite{Cheng}, it equals $2\ord_x(u)$. The following lemma describes its relationship to $\ord_x(u)$ under rather weak regularity assumptions on $u$.
\begin{lemma}
\label{iso2}
Let $(M,g)$ be a smooth compact Riemannian surface, possibly with boundary, and $u$ be a non-trivial solution of the Schrodinger equation~\eqref{schro} with  $V\in K^{2,\delta}(M)$, where $0<\delta<1$. Let $x\in\mathcal N^2(u)$ be an isolated point in $\mathcal N^2(u)$ such that the degree $\deg(x)$ is finite. Then $\deg(x)$ is an even integer that is at least $2\ord_x(u)$.
\end{lemma}
\begin{proof}
Denote by $N$ the vanishing order $\ord_x(u)$, that is the degree of an approximating homogeneous harmonic polynomial $P_N(y-x)$, see Prop.~\ref{vo}. Choose a sufficiently small neighbourhood $B$ of $x$ such that it does not contain other points from $\mathcal N^2(u)$ and does not contain nodal loops. We identify $B$ with a unit ball in the Euclidean plane such that the point $x$ corresponds to the origin. By $B_\lambda\subset B$ we mean a neighbourhood that corresponds to a ball of radius $\lambda$, where $0<\lambda<1$. Consider the rescaled function
$$
u_\lambda(y)=\lambda^{-N}u(\lambda\cdot y)
$$
defined on the unit circle $S=\{y:\abs{y}=1\}$. Prop.~\ref{vo} implies that $u_\lambda(y)$ converges uniformly to the homogeneous harmonic polynomial $P_N(y)$ as $\lambda\to 0$, when $y$ ranges over the unit circle $S$. As is known, $P_N(y)$ changes sign on $S$ precisely $2N$ times, and hence, the corresponding zeroes are stable under the perturbation of $P_N(y)$. Thus, we conclude that for all sufficiently small $\lambda>0$ the zeroes of $u_\lambda$ lie in small pair-wise non-intersecting neighourhoods $U_i\subset S$, where $i=1,\ldots,2N$, of the zeroes of $P_N(y)$, and each $U_i$ contains at least one zero of $u_\lambda$. Choosing a sufficiently small $\lambda>0$,  by Corollary~\ref{degree2} we may assume that the zeroes of $u_\lambda$ correspond to the intersections of nodal edges incident to $x$ with $\partial B_\lambda$. Further, the intersections of the nodal edges incident to $x$ with $B_\lambda$ lie in the cones
$$
C_i(\lambda)=\{t\cdot\lambda U_i: 0<t<1\},\qquad\text{where}\quad i=1,\ldots,2N.
$$
Since the cones $C_i(\lambda)$ are pair-wise non-intersecting and each of them contains at least one connected piece incident to $x$ of a nodal edge, we conclude that $\deg(x)$ is at least $2N$. 

Now we claim that each cone $C_i(\lambda)$ contains an odd number of nodal edge pieces incident to $x$, and hence, the degree $\deg(x)$ is an even integer. Indeed, the solution $u$ has different signs on the connected components of $B_\lambda\backslash\cup C_i(\lambda)$ adjacent to the same cone; they coincide with the signs of $u_\lambda$ and the approximating homogeneous harmonic polynomial $P_N$. Since $u$ also has different signs on adjacent nodal domains, each nodal edge piece incident to $x$ contributes to the change of sign, and the claim follows in a straightforward fashion.
\end{proof}

\subsection{Properties of the vanishing order}
The proof of Prop.~\ref{reg} is based on the following improvement of Prop.~\ref{vo} due to~\cite{HoHoNa}, which is important for our sequel considerations. Below we denote by $B$ a coordinate chart viewed as a ball in the Euclidean plane, and by $B_{1/2}$ the ball of twice smaller radius.
\begin{prop}
\label{me}
Let $(M,g)$ be a smooth compact Riemannian surface, possibly with boundary, and $u$ be a non-trivial solution of the Schrodinger equation~\eqref{schro} with  $V\in K^{2,\delta}(M)$, where $0<\delta<1$. Let $B$ be a coordinate chart in the interior of $M$ viewed as a ball in the Euclidean plane.  Then for a sufficiently small $B$ and any $\ell\geqslant 1$ there exists a constant $C>0$ such that for any  point $y\in\mathcal N^\ell(u)\cap B_{1/2}$  there exists a degree $\ell$ homogeneous harmonic polynomial $P_\ell^y$ such that
$$
\abs{u(x)-P_\ell^y(x-y)}\leqslant C(\sup_B\abs{u})\abs{x-y}^{\ell+\delta}\qquad\text{for any~~}x\in B,
$$
and the polynomials $P_\ell^y$ satisfy $\abs{P_\ell^{y}(\bar x)}\leqslant C_*(\sup_B\abs{u})$ for any $\abs{\bar x}=1$, where the constants $C$  and $C_*$ do not depend on a solution $u$. 
\end{prop}
Mention that the harmonic polynomials $P_\ell^y$ above either vanish identically or coincide with approximating harmonic polynomials at $y$ from Prop.~\ref{vo}. The main estimate of Prop.~\ref{me} is stated in~\cite[Theorem~1]{HoHoNa}. The bound for the values of the harmonic polynomials on the unit circle follows from the proof, and is explained explicitly on~\cite[p.1256]{HoHoNa}.

We proceed with studying the vanishing order $\ord_x(u)$ as a function of $x\in M$. The following lemma is a straightforward consequence of Prop.~\ref{me}. We include a proof for the completeness of exposition.
\begin{lemma}
\label{weak_upper}
Let $(M,g)$ be a smooth compact Riemannian surface, possibly with boundary, and $u$ be a non-trivial solution of the Schrodinger equation~\eqref{schro} with  $V\in K^{2,\delta}(M)$, where $0<\delta<1$. Then the function $\ord_x(u)$ is upper semi-continuous in the interior of $M$, that is, for any sequence $x_i$ converging to an interior point $x\in M$ one has the inequality $\lim\sup\ord_{x_i}(u)\leqslant\ord_x(u)$.
\end{lemma}
\begin{proof}
For a proof of the lemma it is sufficient to show that if $x_i$ belong to $\mathcal N^\ell(u)$, then so does the limit point $x$. Without loss of generality, we may assume that the points $x_i$ lie in a coordinate chart $B$ that is identified with a unit ball in $\mathbf R^2$ centered at the origin $x=0$, and $x_i\to 0$ as $i\to +\infty$.  In addition, to simplify the notation, we assume that $\sup\abs{u}$ on $B$ equals $1$. Let $P^i_\ell$ be a degree $\ell$ homogeneous harmonic polynomial corresponding to $x_i$ from Prop.~\ref{me}. Representing $u$ as the sum of $u-P^i_\ell$ and $P^i_\ell$, we obtain
\begin{multline*}
\abs{u(x)}\leqslant \abs{u(x)-P^i_\ell(x-x_i)}+\abs{P^i_\ell(x-x_i)}\\ \leqslant C\abs{x-x_i}^{\ell+\delta}+C_*\abs{x-x_i}^{\ell}\qquad\text{for any~~}x\in B,
\end{multline*}
where the second inequality for a sufficiently large $i$ follows from Prop.~\ref{me}. Passing to the limit as $i\to +\infty$, we get
$$
\abs{u(x)}\leqslant C'\abs{x}^{\ell}\qquad\text{for any~~}x\in B,
$$
and conclude that the vanishing order at the origin is at least $\ell$. 
\end{proof}

Our last lemma says that the vanishing order $\ord_x(u)$ is strictly upper semi-continuous on $\mathcal N^2(u)$.
\begin{lemma}
\label{upper}
Let $(M,g)$ be a smooth compact Riemannian surface, possibly with boundary, and $u$ be a non-trivial solution of the Schrodinger equation~\eqref{schro} with  $V\in K^{2,\delta}(M)$, where $0<\delta<1$. Then for any sequence $x_i\in\mathcal N^2(u)$ converging to an interior point $x\in M$ we have $\lim\sup\ord_{x_i}(u)<\ord_x(u)$.
\end{lemma}
\begin{proof}
As in the proof of Lemma~\ref{weak_upper}, we assume that the points $x_i$ belong to a coordinate chart $B$, viewed as  a unit ball in $\mathbf R^2$ centered at the origin $x=0$,  and $x_i\to 0$ as $i\to+\infty$.  We also suppose that $\sup\abs{u}$ on $B$ equals $1$. First, by Lemma~\ref{weak_upper} we conclude that the upper limit  $\lim\sup\ord_{x_i}(u)$ is finite; we denote it by $N$. After a selection of a subsequence, we may assume that the vanishing order $\ord_{x_i}(u)$ equals $N$ for each $x_i$.  By Lemma~\ref{weak_upper} it remains to show that the vanishing order $\ord_x(u)$ at the origin $x $ can not be equal to $N$.

Suppose the contrary; the order of $u$ at the origin equals $N\geqslant 2$. Let $P_N$ be an approximating homogeneous harmonic polynomial for $u$ at the origin. By Prop.~\ref{me}, for a sufficiently large index $i$ we have
\begin{multline}
\label{diff}
\abs{P_N(x)-P_N^i(x-x_i)}\leqslant\abs{u(x)-P_N(x)}+\abs{u(x)-P^i_N(x-x_i)}\\
\leqslant C(\abs{x}^{N+\delta}+\abs{x-x_i}^{N+\delta}) \qquad\text{for any~~}x\in B,
\end{multline}
where $P^i_N$ is an approximating homogeneous harmonic polynomial at $x_i$. Denote by $\lambda_i$ the absolute value $\abs{x_i}$, and by $\bar x_i$ the point $\lambda_i^{-1}x_i$ on the unit circle. Setting $x=\lambda_i\bar x$ in inequality~\eqref{diff} and using the homogeneity of the left hand-side, we obtain
\begin{equation}
\label{scaled_diff}
\abs{P_N(\bar x)-P_N^i(\bar x-\bar x_i)}\leqslant (1+2^{N+\delta})C\lambda_i^\delta\ \qquad\text{for any~~}\abs{\bar x}=1.
\end{equation}
Without loss of generality, we may assume that the sequence $\bar x_i$ converges to a point $\bar x_0$, $\abs{\bar x_0}=1$. Setting $\bar x$ to be equal to $\bar x_i$ in inequality~\eqref{scaled_diff} and passing to the limit as $i\to +\infty$, we see that $\bar x_0$ is a zero of $P_N$. Recall that the nodal set of $P_N$ consists of $n$ straight lines passing through the origin; the vanishing order of the origin equals $N$, and any other nodal point, such as $\bar x_0$, has vanishing order $1$. On the other hand, by Prop.~\ref{me} the polynomials $P^i_N$ are uniformly bounded on the unit circle, and since in polar coordinates they have the form 
$$
a_ir^N\cos (N\theta)+b_ir^N\sin (N\theta),
$$
we conclude that, after a selection of a subsequence, they converge either to zero or to a harmonic homogeneous polynomial $P^0_N$ of degree $N$. If the former case occurs, then after passing to the limit in inequality~\eqref{scaled_diff}, we see that $P_N(x)$ vanishes, and arrive at a contradiction. Now assume that the harmonic polynomials $P^i_N$ converge to a non-trivial harmonic polynomial $P^0_N$. Then the polynomials $P^i_N(\bar x-\bar x_i)$ converge uniformly to $P^0_N(\bar x-\bar x_0)$, and passing to the limit in inequality~\eqref{scaled_diff}, we conclude that $P_N(\bar x)$ coincides identically with $P^0_N(\bar x-\bar x_0)$. Now, since $N\geqslant 2$, it is straightforward to arrive at a contradiction. The polynomial $P_N(\bar x)$ has precisely $2N$ zeroes as $\bar x$ ranges over the unit circle, while the polynomial $P^0_N(\bar x-\bar x_0)$ has at most $N+1$.
\end{proof}
\begin{corollary}
\label{cor_upper}
Let $(M,g)$ be a smooth compact Riemannian surface, possibly with boundary, and $u$ be a non-trivial solution of the Schrodinger equation~\eqref{schro} with  $V\in K^{2,\delta}(M)$, where $0<\delta<1$. Then the set $\mathcal N^2(u)$ is totally disconnected, that is its every non-empty connected subset is a single point. Besides, the complement $\mathcal N(u)\backslash\mathcal N^2(u)$ is open and dense in the nodal set.
\end{corollary}
\begin{proof}
Suppose the contrary to the first statement. Then there exists a non-empty connected subset $C\subset\mathcal N^2(u)$ that is not a single point. Since any point $x\in C$ is the limit of a non-trivial sequence in $C$, by Lemma~\ref{upper} we conclude that $C\subset\mathcal N^\ell(u)$ for any $\ell\geqslant 2$. Hence, the solution $u$ vanishes to an infinite order at  $C$, and by the strong unique continuation, Prop.~\ref{uc}, vanishes identically.  This contradiction demonstrates the first statement.

By Lemma~\ref{weak_upper} the set $\mathcal N^2(u)$ is closed, and for a proof of the second statement of the corollary it remains to show that the complement $\mathcal N(u)\backslash\mathcal N^2(u)$ is dense. Suppose the contrary. Then for some point $p\in\mathcal N(u)$ there exists a ball $B_\varepsilon(p)$ such that $C=B_\varepsilon(p)\cap\mathcal N(u)$ is contained in $\mathcal N^2(u)$. By Harnack inequality~\cite{AS,Si} no point in the nodal set can be isolated, and we conclude that any $x\in C$ is the limit of a non-trivial sequence in $C$. Now we arrive at a contradiction in the fashion similar to the one above.
\end{proof}

\section{Prime ends of nodal domains}
\label{pe}
Now we study the nodal set $\mathcal N(u)$ from the point of view of the topology of nodal domains. More precisely, we describe the structure of prime ends of nodal domains. The notion of prime end goes back to Caratheodory~\cite{Cara}, who used it to describe the behaviour of conformal maps on the boundaries of simply connected domains. Later his theory has been extended to general open subsets in manifolds~\cite{Ep}. However, main applications seem to be restricted to $2$-dimensional problems, see~\cite{Mil}.  We start with recalling the necessary definitions, following closely~\cite{Ep}. 

Let $\Omega\subset M$ be a connected open subset, where we view $M$ as the interior of a compact Riemannian surface. For a subdomain $D\subset\Omega$ by $\partial D$ we mean the interior  boundary, that is
$$
\partial D=\Omega\cap\bar D\cap (\overline{\Omega\backslash D}).
$$
\begin{defin}
A {\em chain} in $\Omega$ is a sequence $\{D_i\}$, $i=1,2,\ldots$, of open connected subsets of $\Omega$ such that:
\begin{itemize}
\item $\partial D_i$ is connected and non-empty for each $i$, and
\item $\bar D_{i+1}\cap\Omega\subset D_i$ for each $i$.
\end{itemize}
Two chains $\{D_i\}$ and $\{D'_i\}$ are called {\em equivalent} if for a given $i$ there exists $j>i$ such that $D'_j\subset D_i$ and $D_j\subset D'_i$.
\end{defin}
\begin{defin}
A chain in $\Omega$ is called the {\em topological chain} if there exists a point $p\in M$ such that:
\begin{itemize}
\item the diameter of $(p\cup\partial D_i)$ tends to zero as $i\to+\infty$, and
\item the distance $\dist(p,\partial D_i)>0$ for each $i$.
\end{itemize}
The point $p$ above is called the {\em principal point} of $\{D_i\}$. A {\em prime point} of $\Omega$ is the equivalence class of a topological chain.
\end{defin}
Clearly, for a given topological chain the principal point $p\in\bar\Omega$ is unique. Mention also that the above definitions do not depend on a metric on $M$. The set of all prime points of $\Omega$ is denoted by $\hat\Omega$. It is made into a topological space by taking the sets $\hat U$, formed by prime points represented by chains $\{D_i\}$ such that each $D_i$ lies in an open subset $U\subset\Omega$, as a topological basis. There is a natural embedding $\omega:\Omega\to\hat\Omega$, defined by sending a point $x\in\Omega$ to the equivalence class of a sequence of concentric balls centered at $x$ whose diameters tend to zero. As is shown in~\cite[Sect.~2]{Ep}, the map $\omega$ embeds $\Omega$ homeomorphically onto an open subset in $\hat\Omega$. A {\em prime end} of $\Omega$ is  a prime  point which is not in $\omega(\Omega)$. A {\em principal point} of a prime end is any principal point of any representative topological chain.

Although a given topological chain has only one principal point, a prime end may have many. The simplest example is given by considering a domain whose boundary has an oscillating behaviour similar to the graph of $\sin(1/x)$. The collection of all principal points is a subset of the {\em impression} $\cap\bar D_i$ of a prime end. The latter does not depend on a representative topological chain, and is a compact connected subset of the boundary $\partial\Omega$. Mention also that a given point $x\in\partial\Omega$ can be a principal point of many different prime ends. We refer to~\cite{Ep,Mil} for examples and other details.

The following statement, proved in~\cite[Sect.~6]{Ep}, shows that prime ends give a useful compactification (the so-called {\em Caratheodory compactification}) of open subdomains.
\begin{prop}
\label{compact}
Let $(M,g)$ be a Riemannian surface, viewed as the interior of a compact surface, and $\Omega\subset M$ be a connected open subset such that the first homology group $H_1(\Omega,{\mathbf Q} )$ is finite-dimensional. Then there is a homeomorphism of $\hat\Omega$ onto a compact surface with boundary that maps the set of prime ends onto its boundary. 
\end{prop}
We proceed with studying properties of nodal sets. The following lemma says that all prime ends of nodal domains have the simplest possible structure: any of them has only one principal point that coincides with its impression.
\begin{lemma}
\label{ends}
Let $(M,g)$ be a smooth compact Riemannian surface, possibly with boundary. Let $u$ be a non-trivial solution to the Schrodinger equation~\eqref{schro} with a potential $V\in K^{2,\delta}(M)$, where $0<\delta<1$, and $\Omega$ be its nodal domain. Then for any prime end $[D_i]$ of $\Omega$ its impression $\cap\bar D_i$ consists of a single point. In particular, any prime end has only one principal point.
\end{lemma}
\begin{proof}
First, the statement holds for any prime end that has a principal point $x$ in the complement $\mathcal N(u)\backslash\mathcal N^2(u)$. Indeed, then the point $x$ belongs to a nodal edge, which is the image of $C^1$-smooth regular path, see Prop.~\ref{reg}. By the implicit function theorem we can view a small nodal arc containing $x$ as a line segment in $\mathbf{R}^2$. Then it is straightforward to see that any chain that has $x$ as a principal point is equivalent to a chain that consists of concentric semi-disks centered at $x$ whose diameters converge to zero. Its impression consists of the point $x$ only. 

Now suppose that a given prime end has a principal point $x\in\mathcal N^2(u)$. Then we claim that its impression $I$ does not have any points in $\mathcal N(u)\backslash\mathcal N^2(u)$. Suppose the contrary. Then, since the impression $I$ of a prime end is connected, we conclude that $I$ contains a non-trivial arc $C$ that belongs to some nodal edge; that is, $C$ is a connected subset of $\mathcal N(u)\backslash\mathcal N^2(u)$ that is not a single point, and $\dist(x,C)>0$. Let $\{D_i\}$ be a representative topological chain whose principal point is $x$, and $E_i$ be the set $\partial D_i\backslash I$, where $\partial D_i$ is the boundary of $D_i$ viewed as a subset in $M$. First, it is straightforward to see that for any $y\in C\subset I$ the distance $\dist (y,E_i)$ converges to zero as $i\to +\infty$. Indeed, for otherwise there is a neighbourhood $U$ of $y$ in $\bar D_i$ such that $U\subset\bar D_i$ for any $i$. More precisely, viewing $C$ around $y$ as a straight segment in $\mathbf{R}^2$, we may choose $U$ to be diffeomorphic to a semi-disk $B^+_\varepsilon(y)$, assuming that $\dist(y,E_i)\geqslant 2\varepsilon$. Then we obtain the inclusions $U\subset I\subset\partial\Omega$, which are impossible. Thus, we see that any point $y\in C$ is the limit of a sequence $y_i\in\bar E_i$. Indeed, as $y_i$ one can take a point at which the distance $\dist(y,E_i)$ is attained. This implies that there is a sequence $C_i\subset\bar E_i$ of subsets that converges to a nodal arc $C$ in the Hausdorff distance. Clearly,  the sets $E_i\backslash (\partial D_i\cap\Omega)$ lie in the nodal set $\mathcal N(u)$, and since the interior boundaries $\partial D_i\cap\Omega$ converge to the point $x$, we conclude that for a sufficiently large $i$ the subset $C_i$ lies in the nodal set. Further, since the set $\mathcal N(u)\backslash\mathcal N^2(u)$ is open in the nodal set (see Lemma~\ref{weak_upper}), we see that each $C_i$ lies in $\mathcal N(u)\backslash\mathcal N^2(u)$. Thus, without loss of generality, we may assume that $C_i$ are arcs of nodal edges. Combining the latter with Prop.~\ref{reg}, or following the argument in the proof of Corollary~\ref{degree}, we arrive at a contradiction.

Thus, the impression $I$ does not have points in the complement $\mathcal N(u)\backslash\mathcal N^2(u)$, and is contained in $\mathcal N^2(u)$. By Corollary~\ref{cor_upper} the set $\mathcal N^2(u)$ is totally disconnected, and since the impression $I$ is connected, it has to coincide with the point $x$.
\end{proof}
\begin{corollary}
\label{top}
Under the hypotheses of Lemma~\ref{ends}, the following statements hold:
\begin{itemize}
\item[(i)] any point $x\in\partial\Omega$ is accessible, that is it can be joined with any interior point in $\Omega$ by a continuous path $\gamma:[0,1]\to M$ such that $\gamma(0)=x$ and the image $\gamma(0,1]$ lies in $\Omega$;
\item[(ii)] for any point $x\in\partial\Omega$ and any sufficiently small neighbourhood $U$ of $x$ there are only finitely many connected components $U_1,\ldots,U_k$ of $\Omega\cap U$ such that $x\in\bar U_i$ and the union $\cup\bar U_i$ is a neighbourhood of $x$ in $\bar\Omega$;
\item[(iii)] the boundary $\partial\Omega$ is locally connected.
\end{itemize}
\end{corollary}
\begin{proof}
We derive the statements using the results in~\cite{Ep}, which apply to open domains $\Omega\subset M$ whose first homology group $H_1(\Omega,{\mathbf Q} )$ is finite-dimensional. Mention that all statements are local, and hold trivially for the boundary points $x\in\mathcal N(u)\backslash\mathcal N^2(u)$. To prove the corollary for the boundary points $x\in\mathcal N^2(u)$ we may assume, after cutting $\Omega$ along smooth simple closed paths, that $\Omega$ has zero genus. Moreover, after cutting along paths joining points from $\mathcal N(u)\backslash\mathcal N^2(u)$ on different boundary components of $\Omega$, we may assume that $\Omega$ is simply connected, and the results in~\cite{Ep} apply.

In more detail, the first statement is a consequence of Lemma~\ref{ends}, \cite[Theorem~7.4]{Ep}, and~\cite[Theorem~8.2]{Ep}. The second statement follows from Lemma~\ref{ends} and~\cite[Theorem~8.2]{Ep}, and the third from Lemma~\ref{ends} and~\cite[Theorem~8.3]{Ep}.
\end{proof}

\section{The proofs}
\label{proofs}
\subsection{Proof of Theorem~\ref{t2}}
Let $(M,g)$ be a compact Riemannian surface, and $u$ be a solution to the Schrodinger equation~\eqref{schro} with a potential $V\in K^{2,\delta}(M)$, where $0<\delta<1$. First, we intend to generalise Theorem~\ref{t2} to certain subdomains $\Omega\subset M$. 
\begin{defin}
\label{proper}
A connected open subset $\Omega\subset M$ is called the {\em proper subdomain} with respect to a solution $u$ if its boundary consists of finitely many connected components, and the solution $u$ has finitely many nodal domains in $\Omega$, that is, the number of connected components $\Omega\backslash\mathcal N(u)$ is finite.
\end{defin}
If $u$ is an eigenfunction, then by Courant's nodal domain theorem the surface $M$ itself is a proper subdomain with respect to $u$. However, for our method it is also important to consider proper subdomains whose closures are contained in the interior of $M$. The hypothesis on the finite number of boundary components guarantees that a domain $\Omega$ has finite topology, and by Prop.~\ref{compact}, is homeomorphic to the interior of a compact surface with boundary. The second hypothesis in Definition~\ref{proper} mimics an important property of eigenfunctions, and is essential for our sequel arguments. Below by $\mathcal N_\Omega(u)$ and $\mathcal N^\ell_\Omega(u)$ we denote the sets $\mathcal N(u)\cap\Omega$ and $\mathcal N^\ell(u)\cap\Omega$ respectively.

Theorem~\ref{t2} is a consequence of the following more general result.
\begin{theorem}
\label{t3}
Let $(M,g)$ be a compact Riemannian surface, possibly with boundary, and $u$ be a non-trivial solution to the Schrodinger equation~\eqref{schro} with a potential $V\in K^{2,\delta}(M)$, where $0<\delta<1$. Then for any proper subdomain $\Omega\subset M$ with respect to $u$ the set $\mathcal N_\Omega^2(u)$ is finite, and the complement $\mathcal N_\Omega(u)\backslash\mathcal N^2(u)$ has finitely many connected components. Moreover, for any $x\in\mathcal N_\Omega^2(u)$ the number of connected components of  $\mathcal N_\Omega(u)\backslash\mathcal N^2(u)$ incident to $x$ (if one connected component starts and ends at $x$, then it counts twice) is an even integer that is at least $2\ord_x(u)$.
\end{theorem}

The proof of Theorem~\ref{t3} is based on the two lemmas below. The first lemma shows that proper neighbourhoods form a topological basis at any point $x\in\Omega$. Its proof relies on the topological consequences of our study of prime ends in Sect.~\ref{pe}.
\begin{lemma}
\label{locality}
Under the hypotheses of Theorem~\ref{t3}, for any point $x\in\mathcal N_\Omega(x)$ and any sufficiently small ball $B_\varepsilon(x)$ centered at $x$ there exists a proper subdomain $U_\varepsilon(x)$ with respect to $u$ such that $x\in U_\varepsilon(x)\subset B_\varepsilon(x)$. 
\end{lemma}
\begin{proof}
Let $x\in\mathcal N(u)$ be an interior nodal point in $\Omega$, and $\Omega_1,\ldots,\Omega_m$ be a collection of all nodal domains whose closure contains $x$. By Corollary~\ref{top} for any sufficiently small open ball $B_\varepsilon(x)\subset\Omega$ there are only finitely many connected components $\Omega_i^j$, $j=1,\ldots, r_i$, of the intersection $B_\varepsilon(x)\cap\Omega_i$ whose closure contains $x$. Besides, the union $F_i=\cup_j\bar\Omega_i^j$ is a neighbourhood of $x$ in $\bar\Omega_i$. Thus, we conclude that the set $U_\varepsilon(x)=\Int(\cup F_i)$ contains $x$. Clearly, the connected components of the complement $U_\varepsilon(x)\backslash\mathcal N(u)$ are precisely the domains $\Omega_i^j$, and it remains to show that $U_\varepsilon(x)$ has finitely many boundary components. Choosing $\varepsilon>0$ such that the metric ball $B_\varepsilon(x)$ is homeomorphic to a ball in $\mathbf{R}^2$, it is straightforward to see that any boundary component of $U_\varepsilon(x)$ that lies in $B_\varepsilon(x)$ bounds a union of nodal domains. Since the number of nodal domains is finite, then choosing $\varepsilon>0$ even smaller we conclude that $U_\varepsilon(x)$ is simply connected, and hence, its boundary is connected. Thus, the neighbourhood $U_\varepsilon(x)$ is indeed a proper subdomain with respect to a solution $u$.
\end{proof}

The second lemma says that if the set $\mathcal N^2_\Omega(u)$ consists of isolated points, then it is necessarily finite, and the nodal set has the structure of a finite graph with the vertex set $\mathcal N^2_\Omega(u)$.
\begin{lemma}
\label{iso}
Under the hypotheses of Theorem~\ref{t3}, suppose that the set $\mathcal N^2_\Omega(u)$ consists of isolated points.
Then the set $\mathcal N_\Omega^2(u)$ is finite, and the complement $\mathcal N_\Omega(u)\backslash\mathcal N^2(u)$ has finitely many connected components. 
\end{lemma}
The proof of the last lemma appears at the end of the section. Now we proceed with the proof of Theorem~\ref{t3}.

\begin{proof}[Proof of Theorem~\ref{t3}]
By Lemma~\ref{iso} for a proof of the theorem it is sufficient to show that the set $\mathcal N_\Omega^2(u)$ consists of isolated points in $\Omega$. The second statement of the theorem is a direct consequence of Lemma~\ref{iso2}. First, we consider the case of proper subdomains $\Omega\subset M$ whose closures are contained in the interior of $M$, $\bar\Omega\subset M$. Given such a subdomain $\Omega$, it is straightforward to see that the {\em maximal vanishing order} $\ell=\max\{\ord_x(u)\}$, where $x\in\Omega$, is finite. Indeed, for otherwise there exists a point $p\in\bar\Omega$ that is the limit of points $x_i\in\Omega$ such that $\ord_{x_i}(u)\to +\infty$ as $i\to +\infty$. Then, by Lemma~\ref{weak_upper}, the solution $u$ vanishes to an infinite order at $p$, and the strong unique continuation, Prop.~\ref{uc}, implies that $u$ vanishes identically. 

Let $\Omega\subset M$ be a proper subdomain whose closure is contained in the interior of $M$. We prove that the set $\mathcal N^2_\Omega(u)$ is finite by induction in the maximal vanishing order $\ell$. Clearly, the statement holds for all solutions $u$ and proper subdomains $\Omega$ such that the maximal vanishing order  equals $2$. Indeed, in this case by Lemma~\ref{upper} the set $\mathcal N^2_\Omega(u)$ consists of isolated points, and by Lemma~\ref{iso}, is finite. Now we perform an induction step. Suppose that the set $\mathcal N^2_\Omega(u)$ is finite for all solutions $u$ to the Schrodinger equation~\eqref{schro} on $M$ and all proper subdomains $\Omega$ , whose closure is contained in the interior of $M$, such that
$$
\max\{\ord_x(u): x\in\Omega\}\leqslant \ell-1.
$$
Now let $u$ be a solution on $M$ and $\Omega$ be a proper subdomain such that  the maximal vanishing order equals $\ell$,
$$
\max\{\ord_x(u): x\in\Omega\}=\ell.
$$
By Lemma~\ref{upper} the set $\mathcal N^\ell_\Omega(u)$ consists of isolated points in $\Omega$. Pick a point $p\in\mathcal N^2_\Omega(u)$. By Lemma~\ref{locality} there is its neighbourhood $U$ that is a proper subdomain such that $\bar U\subset\Omega$. Then the neighbourhood $U$ may contain only finitely many points $p_1, p_2,\ldots, p_m$ whose vanishing order equals $\ell$. Since the domain $U_0=U\backslash\{p_1,\ldots, p_m\}$ is proper with respect to $u$, then the induction hypothesis implies that the set $\mathcal N^2(u)\cap U_0$ is finite. Hence, so is the set $\mathcal N^2(u)\cap U$.  Thus, we conclude that $\mathcal N^2_\Omega(u)$ consists of isolated points in $\Omega$, and by Lemma~\ref{iso}, is finite.

The statement that the set $\mathcal N^2_\Omega(u)$ consists of isolated points in $\Omega$ for an arbitrary proper subdomain $\Omega\subset M$ follows directly from the case considered above together with Lemma~\ref{locality}. 
\end{proof}

\subsection{Proof of Theorem~\ref{t1}}
Now we show how Theorem~\ref{t2} implies the multiplicity bounds. We give an argument following the strategy described in~\cite[Sect. 6]{KKP}.  It relies on two lemmas that appear below. The first lemma gives a lower bound for the number of nodal domains via the vanishing order of points $x\in\mathcal N^2(u)$.
\begin{lemma}
\label{nnd}
Under the hypotheses of Theorem~\ref{t1}, for any non-trivial eigenfunction $u$ of an eigenvalue $\lambda_k(g,\nu)$ the number of its nodal domains is at least $\sum(\ord_x(u)-1)+\chi+l$, where the sum is taken over all points in $\mathcal N^2(u)$ and $\chi$ and $l$ stand for the Euler-Poincare number and the number of boundary components of $M$ respectively.
\end{lemma}
Before giving a proof we introduce some notation that is useful in sequel. First, by Theorem~\ref{t3} the nodal set $\mathcal N(u)$ of any eigenfunction $u$ on $M$ can be viewed as a finite graph, called {\em nodal graph}. Its vertices are points in $\mathcal N^2(u)$ and the edges are connected components of $\mathcal N(u)\backslash\mathcal N^2(u)$. Below we denote by $\bar M$ a closed surface, viewed as the image of $M$ under collapsing its boundary components to points, and by $\bar\chi$ its Euler-Poincare number.  Let $\bar{\mathcal N}(u)$ be the corresponding image of a nodal graph $\mathcal N(u)$, called the {\em reduced nodal graph}. Its edges are the same nodal arcs, and there are two types of vertices: vertices that correspond to the boundary components that contain limit points of nodal lines, called {\em boundary component vertices}, and genuine vertices that correspond to the points in $\mathcal N^2(u)$, called {\em interior vertices}. By {\em faces} of the graph $\bar{\mathcal N}(u)$ we mean the connected components of the complement $\bar M\backslash\bar{\mathcal N}(u)$. Clearly, they can be identified with the nodal domains of an eigenfunction $u$. 

\begin{proof}[Proof of Lemma~\ref{nnd}]
Let $\bar{\mathcal N}(u)$ be a reduced nodal graph in $\bar M$. By Theorem~\ref{t3} it is a finite graph, and let $v$, $e$, and $f$ be the number of its vertices, edges, and faces respectively. We also denote by $r$ the number of boundary component vertices in $\bar{\mathcal N}(u)$. Recall that the number of edges satisfies the relation $2e=\sum\deg(x)$, where the sum is taken over all vertices. Since an eigenfunction $u$ has different signs on adjacent nodal domains, the degree of each boundary component vertex is at least two, and we obtain
$$
e\geqslant r+\frac{1}{2}\sum\deg(x) \geqslant r+\sum\ord_x(u),
$$
where the sum is taken over all interior vertices $x\in\mathcal N^2(u)$. The second inequality above follows from the relation $\deg(x)\geqslant 2\ord_x(u)$, see Theorem~\ref{t3}. Viewing the number of vertices $v$ as the sum $r+\sum 1$, where the sum symbol is again taken over $x\in\mathcal N^2(u)$, by the Euler inequality~\cite[p. 207]{Gib} we have
$$
f\geqslant e-v+\bar\chi\geqslant\sum (\ord_x(u)-1)+\bar\chi,
$$
where $\bar\chi=\chi+l$ is the Euler-Poincare number of $\bar M$. Since $f$ is precisely the number of nodal domains, we are done.
\end{proof}

We proceed with the second lemma. In the case when the potential of a Schrodinger equation is smooth it is due to~\cite{Na}, see also~\cite{KKP}. The proof relies essentially on Prop.~\ref{vo}.

\begin{lemma}
\label{nn}
Let $(M,g)$ be a compact Riemannian surface, possibly with boundary, and $u_1,\ldots,u_{2n}$ be a collection of non-trivial linearly independent solutions to the Schrodinger equation~\eqref{schro} with a potential $V\in K^{2,\delta}(M)$, where $0<\delta<1$. Then for a given interior point $x\in M$ there exists a non-trivial linear combination $u=\sum\alpha_iu_i$ whose vanishing order $\ord_x(u)$ at the point $x$ is at least $n$.
\end{lemma}
\begin{proof}
Let $V$ be a linear space spanned by the functions $u_1,\ldots, u_{2n}$, and $V_i$ be its subspace formed by solutions $u\in V$ whose vanishing order at $x$ is at least $i$, $\ord_x(u)\geqslant i$. Clearly, the subspaces $V_i$ form a nested sequence, $V_{i+1}\subset V_i$. The statement of the lemma says that $V_n$ is non-trivial. Suppose the contrary: the subspace $V_n$ is trivial. Then, it is straightforward to see that the dimension of $V$ satisfies the inequality
$$
\dim V\leqslant 1+\sum_{i=1}^{n-1}\dim (V_i/V_{i+1});
$$
the equality occurs if the space $V$ does not coincide with $V_1$. By Prop.~\ref{vo} the factor-space $V_i/V_{i+1}$ can be identified with a subspace of homogeneous harmonic polynomials on $\mathbf R^2$ of degree $i$. When the degree $i\geqslant 1$, the space of such polynomials has dimension two, and we obtain
$$
\dim V\leqslant 1+2(n-1)=2n-1.
$$
Thus, we arrive at a contradiction with the hypotheses of the lemma.
\end{proof}

Now we finish the proof of Theorem~\ref{t1}. Suppose the contrary to its statement. Then there exists at least $2(2-\chi-l)+2k+2$ linearly independent eigenfunctions corresponding to the eigenvalue $\lambda_k(\mu,g)$. Pick an interior point $x\in M$. By Lemma~\ref{nn} there exists a new eigenfunction $u$ whose vanishing order at the point $x$ is at least $2-\chi-l+k+1$. Now the combination with Lemma~\ref{nnd} implies that the number of the nodal domains of $u$ is at least $k+2$. Thus, we arrive at a contradiction with Courant's nodal domains theorem. \qed

\subsection{Proof of Lemma~\ref{iso}}
Since the set $\mathcal N_\Omega^2(u)$ consists of isolated points, we can view the nodal set $\mathcal N_\Omega(u)$ as a graph: the vertices are points in $\mathcal N_\Omega^2(u)$, and the edges are connected components of $\mathcal N_\Omega(u)\backslash\mathcal N_\Omega^2(u)$. Recall that the {\em degree} $\deg(x)$ of a vertex $x\in\mathcal N_\Omega^2(u)$ is defined as the number of edges incident to $x$; if one edge starts and ends at $x$, then it counts twice. The following lemma says that the degree of each vertex has to be finite.
\begin{lemma}
\label{iso1}
Under the hypotheses of Theorem~\ref{t3}, suppose that the set $\mathcal N^2_\Omega(u)$ consists of isolated points. Then the degree $\deg(x)$ of any point $x\in\mathcal N^2_\Omega(u)$ is finite.
\end{lemma}
\begin{proof}
By Corollary~\ref{degree} it is sufficient to show that the number of nodal loops that start and end at a given point $x\in\mathcal N^2_\Omega(u)$ is finite. Suppose the contrary, that is the number of such nodal loops is infinite. Let $\bar\Omega$ be a compactification of $\Omega$, obtained by adding one point for each boundary component. By Prop.~\ref{compact} it is homeomorphic to a closed surface, and we denote by $\bar\chi$ its Euler-Poincare number. Let $\Gamma$ be a subgraph in the nodal graph formed by one vertex $x$ and $m+2-\bar\chi$ nodal loops that start and end at $x$, where $m$ is the number of nodal domains of $u$ in $\Omega$. Denote by $v=1$, $e=m+2-\bar\chi$, and $f$ the number of vertices, edges, and faces of $\Gamma$ respectively. Here by the faces of $\Gamma$ we mean the connected components of $\bar\Omega\backslash\Gamma$. Clearly, they are unions of nodal domains, and $f\leqslant m$. On the other hand, viewing $\Gamma$ as a graph in $\bar\Omega$, by Euler's inequality~\cite[p.~207]{Gib}, we obtain
$$
f\geqslant e-v+\bar\chi=m+1.
$$
This contradiction demonstrates the lemma.
\end{proof}

Now we prove the statement of Lemma~\ref{iso}: the set $\mathcal N_\Omega^2(u)$ is finite, and the complement $\mathcal N_\Omega(u)\backslash\mathcal N^2(u)$ has finitely many connected components. The argument below is based on the results in Sect.~\ref{prem}, and is close in the spirit to the one in~\cite[Sect.~3]{KKP}.

Let $\bar\Omega$ be a closed surface obtained by collapsing boundary components of $\Omega$ to points. By $\bar{\mathcal N}_\Omega$ we denote the {\em reduced nodal graph} in $\bar\Omega$, defined in the proof of Theorem~\ref{t1}. Recall that its edges are the same nodal edges, and there are two types of vertices: vertices that correspond to the boundary components of $\Omega$ that contain limit points of nodal lines, called {\em boundary component vertices}, and genuine vertices that correspond to the points in $\mathcal N^2_\Omega(u)$, called {\em interior vertices}. For a proof of the lemma it is sufficient to show that $\bar{\mathcal N}_\Omega(u)$ is a finite graph. Our strategy is to show that:
\begin{itemize}
\item[(i)] each boundary component vertex has a finite degree and
\item[(ii)] the number of interior vertices is finite in $\Omega$.
\end{itemize}
We are going to construct new graphs in $\bar\Omega$ by {\em resolving  interior vertices} in the following fashion. Let $x\in\mathcal N^2_\Omega(u)$ be an interior vertex. By Lemma~\ref{iso1} its degree is finite, and by Lemma~\ref{iso2} it is an even integer $2n$. Let $B$ be a small disk centered at $x$ that does not contain other vertices. By Corollary~\ref{degree2} we may assume that non-incident to $x$ nodal edges lie in the complement $\Omega\backslash B$. Besides, since the degree is finite, we may also assume  that each nodal loop incident to $x$ intersects $\partial B$ in at least two points. Consider the intersections of nodal edges with $B$, and let $\Gamma_i$, where $i=0,\ldots, 2n-1$, be their connected components incident to $x$. Pick points $y_i\in\bar\Gamma_i\cap\partial B$; one for each $i=0,\ldots,2n-1$. By the resolution of a vertex $x$ we mean a new graph obtained by removing sub-arcs between $x$ and $y_i$ in each nodal edge incident to $x$ and rounding-off them by non-intersecting arcs in $B$ joining the points $y_{2j}$ and $y_{2j+1}$. If there was an edge that starts and ends at $x$, then such a procedure may make it into a loop. We remove all such loops, if they occur.  A new graph, obtained by the resolution of one vertex, has one vertex less and at most as many faces as the original graph.

\noindent
{\em Proof of~(i).} Suppose the contrary. Let us resolve all interior vertices in $\bar{\mathcal N}_\Omega(u)$ in the way described above. The result is a graph $\Gamma$ whose only vertices are boundary component vertices in $\bar{\mathcal N}_\Omega(u)$; denote by $v$ their number. Besides, it has at most as many faces as $\bar{\mathcal N}_\Omega (u)$, that is not greater than the number of nodal domains. Since there is a boundary component vertex in $\bar{\mathcal N}_\Omega(u)$ whose degree is infinite, the same vertex has an infinite degree in $\Gamma$. Let us remove all edges in $\Gamma$ except for at least $v+m+1-\bar\chi$ of them, where $m$ is the number of nodal domains and $\bar\chi$ is the Euler-Poincare number of $\bar\Omega$. The result is a finite graph; it has precisely $v$ vertices, and we denote by $e$ and $f$ the number of its edges and faces respectively. By Euler's inequality, we obtain
$$
f\geqslant e-v+\bar\chi=m+1.
$$
On the other hand, since removing an edge does not increase the number of faces, we have $f\leqslant m$.
Thus, we arrive at a contradiction.

\noindent
{\em Proof of~(ii).} Suppose the contrary, and let $v$ be a number of boundary component vertices in $\bar{\mathcal N}_\Omega(u)$. Let us resolve all interior vertices except for $v+m+1-\bar\chi$ of them. The result is a finite graph; we denote by $v'$, $e'$, and $f'$ the number of its vertices, edges, and faces respectively. Clearly, we have
$$
v'\leqslant 2v+m+1-\bar\chi\quad\text{and}\quad e'\geqslant 2(v+m+1-\bar\chi),
$$
where in the second inequality we used Lemma~\ref{iso2}, saying that the degree of each vertex $x\in\mathcal N_\Omega^2(u)$ is at least $4$. Combining these inequalities with the Euler inequality, we obtain
$$
f'\geqslant e'-v'+\bar\chi\geqslant m+1.
$$
On the other hand, we have $f'\leqslant m$. Thus, we arrive at a contradiction.
\qed

\section{Eigenvalue problems on singular Riemannian surfaces}
\label{sing}
\subsection{Eigenvalue problems on surfaces with measures}
The purpose of this section is to discuss multiplicity bounds on singular Riemannian surfaces. We start with recalling a useful general setting of eigenvalue problems on surfaces with measures, following~\cite{GK}. 

Let $(M,g)$ be a compact Riemannian surface, possibly with boundary, and $\mu$ be a finite absolutely continuous (with respect to $\mathit{dVol}_g$) Radon measure on $M$ that satisfies the {\em decay condition}
\begin{equation}
\label{cem}
\mu(B(x,r))\leqslant Cr^\delta,\qquad\text{for any}\quad r>0\text{ and }x\in M, 
\end{equation}
and some constants $C$ and $\delta>0$. Denote by $L_2^1(M,\mathit{Vol}_g)$ the space formed by distributions whose derivatives are in $L_2(M,\mathit{Vol}_g)$. Then by the results of Maz'ja~\cite{Ma}, see also~\cite{GK}, the embedding
$$
L_2(M,\mu)\cap L_2^1(M,\mathit{Vol}_g)\subset L_2(M,\mu)
$$
is compact, the Dirichlet form $\int\abs{\nabla u}^2\mathit{dVol}_g$ is closable in $L_2(M,\mu)$, and its spectrum is discrete. We denote by
$$
0=\lambda_0(g,\mu)<\lambda_1(g,\mu)\leqslant\ldots\lambda_k(g,\mu)\leqslant\ldots
$$
the corresponding eigenvalues, and by $m_k(g,\mu)$ their multiplicities. As above, we always suppose that the Dirichlet boundary hypothesis is imposed, if the boundary of $M$ is non-empty. The eigenfunctions corresponding to an eigenvalue $\lambda_k(g,\mu)$ are  distributional solutions to the Schrodinger equation
\begin{equation}
\label{eigenf}
-\Delta_gu=\lambda_k(g,\mu)\mu u\qquad\text{on~ }M.
\end{equation}
The latter fact ensures that the analysis in Sect.~\ref{prem}-\ref{proofs} carries over to yield the following result.
\begin{theorem}
\label{t4}
Let $(M,g)$ be a smooth compact Riemannian surface, possibly with boundary, endowed with a finite absolutely continuous Radon measure $\mu$ that satisfies hypothesis~\eqref{cem}. Then the multiplicity $m_k(g,\mu)$ of a Laplace eigenvalue $\lambda_k(g,\mu)$ satisfies the inequality
$$
m_k(g,\mu)\leqslant 2(2-\chi-l)+2k+1\qquad\text{for any}\quad k=1,2,\ldots,
$$
where $\chi$ stands for the Euler-Poincare number of $M$ and $l$ is the number of boundary components.
\end{theorem}
\begin{proof}
First, we claim that the decay hypothesis~\eqref{cem} on the measure $\mu$ implies that its density belongs to the space $K^{2,\delta'}(M)$ for some $0<\delta'<\delta$. Indeed, by Fubini's theorem and the change of variable formula, we obtain
\begin{multline*}
\int\limits_{B(x,r)}\abs{x-y}^{-\delta'}d\mu=\int\limits_{r^{-\delta'}}^{+\infty}\mu\{(y:\abs{x-y}^{-\delta'}\geqslant t)\}dt=\int\limits_{r^{-\delta'}}^{+\infty}\mu(B(x,t^{-1/{\delta'}}))dt\\
=\delta'\int\limits_0^r s^{-\delta'-1}\mu(B(x,s))ds\leqslant C\delta'\int\limits_0^r s^{\delta-\delta'-1}ds.
\end{multline*}
Second, using a variational characterisation of eigenvalues $\lambda_k(g,\mu)$, it is also straightforward to check that the standard proof of Courant's nodal domains theorem carries over for eigenfunctions $u$, which satisfy~\eqref{eigenf}. Hence, Theorem~\ref{t3} applies, and then the argument in the proof of Theorem~\ref{t1} carries over directly to yield the multiplicity bounds.
\end{proof}
Note that, since the Dirichlet energy is conformally invariant, if the measure $\mu$ is the volume measure of a metric $h$ conformal to $g$, then the quantities $\lambda_k(g,\mu)$ are precisely the Laplace eigenvalues of a metric $h$. More generally, the eigenvalue problems on surfaces with singular metrics can be also often viewed as particular instances of the setting of eigenvalues on measures. Below we discuss this point of view in more detail.

Let $(M,g)$ be a Riemannian surface, and $h$ be a Riemannian metric of finite volume defined on the set $M\backslash S$, where $S$ is a closed nowhere dense subset of zero measure. Here the set $S$ plays the role of a singular set of $h$ on $M$. Denote by $\mu$ the volume measure of the metric $h$. In the literature, see e.g.~\cite{Chee}, the Dirichlet spectrum of a singular metric $h$ is normally defined as  the spectrum of the Dirichlet form 
\begin{equation}
\label{d:form}
u\longmapsto\int_{M\backslash S}\abs{\nabla u}^2d\mathit{Vol}_h
\end{equation}
defined on the space $\mathcal C\subset L_2(M,\mu)$ formed by smooth compactly supported functions in $M\backslash S$. Suppose that the set $S$ has zero Dirichlet capacity, the metric $h$ is conformal on $M\backslash S$ to the metric $g$, and its volume measure $\mu$ satisfies the decay hypothesis~\eqref{cem}. Then, it is straightforward to see that the  spectrum of $h$ is discrete and coincides with the set of eigenvalues $\lambda_k(g,\mu)$ defined above. Moreover, the construction makes sense even if a metric $h$ is not smooth on $M\backslash S$ as long as the Dirichlet form~\eqref{d:form} is well-defined. Theorem~\ref{t4} gives multiplicity bounds for such eigenvalue problems. We end with discussing two examples: metrics with conical singularities and, more generally, Alexandrov surfaces of bounded integral curvature.
\subsection{Example~I: metrics with conical singularities}
Let $M$ be a closed smooth surface, and $h$ be a metric on $M$ with a number of conical singularities. Recall that a point $p\in M$ is called the {\em conical singularity} of order $\alpha>-1$ (or angle $2\pi(\alpha+1)$) if in an appropriate local complex coordinate the metric $h$ has the form $\abs{z}^{2\alpha}\rho(z)\abs{dz}^2$, where $\rho(z)>0$. In other words, near $p$ the metric is conformal to the Euclidean cone of total angle $2\pi(\alpha+1)$. As is known, such a metric $h$ is conformal to a genuine Riemannian metric $g$ on $M$ away from the singularities. If a surface $M$ has a non-empty boundary, we do not exclude an infinite number of conical singularities accumulating to the boundary, and suppose that the volume measure $\mathit{Vol}_h$ satisfies the decay hypothesis~\eqref{cem}. For a surface with a finite number of conical singularities the hypothesis on the volume measure is always satisfied. The Dirichlet integral with respect to the metric $h$ is defined as an improper integral; by the conformal invariance, it satisfies the relation
$$
\int_M\abs{\nabla u}_h^2\mathit{dVol}_h=\int_M\abs{\nabla u}_g^2\mathit{dVol}_g
$$
for any smooth function $u$. Thus, we conclude that the Laplace eigenvalues and their multiplicities of a metric $h$ coincide with the quantities $\lambda_k(g,\mathit{Vol}_h)$ and $m_k(g,\mathit{Vol}_h)$, defined above,
and Theorem~\ref{t4} yields the multiplicity bounds. Mention that if a metric $h$ has only a finite number of conical singularities, then the  multiplicity bounds can be also obtained from arguments in~\cite{KKP}.

\subsection{Example~II: Alexandrov surfaces of bounded integral curvature}
The most significant class of surfaces, illustrating our approach, is formed by the so-called Alexandrov surfaces of bounded integral curvature.  Below we recall this notion and give a brief outline of its relevance to our setting; more details and references on the subject can be found in the surveys~\cite{Re93,Tro}. Eigenvalue problems on Alexandrov surfaces of bounded integral curvature are treated in detail in~\cite{GK2}.

\begin{defin*}
A metric space $(M,d)$, where $M$ is a compact smooth surface, is called the {\em Alexandrov surface of bounded integral curvature} if:
\begin{itemize}
\item[(i)] the topology induced by $d$ coincides with the original surface topology on $M$;
\item[(ii)] the metric space $(M,d)$ is a {\em geodesic length space}, that is any two points $x$ and $y\in M$ can be joined by a path whose length is $d(x,y)$;
\item[(iii)] the metric  $d$ is a $C^0$-limit of distances of smooth Riemannian metrics $g_n$ on $M$ whose integral curvatures are bounded, that is
$$
\sup_n\int_M\abs{K_{g_n}}d\mathit{Vol}_{g_n}<+\infty,
$$
where $K_{g_n}$ stands for the Gauss curvature of a metric $g_n$.
\end{itemize}
\end{defin*}
This is a large class of singular surfaces that contains, for example, all polyhedral surfaces as well as surfaces with conical singularities and their limits under the integral curvature bound. The hypothesis~(iii) implies that after a selection of a subsequence the signed measures $K_{g_n}d\mathit{Vol}_{g_n}$ converge weakly to a measure $\omega$ on $M$. By the result of Alexandrov~\cite{AZ}, the measure $\omega$ is an intrinsic characteristic of $(M,g)$; it does not depend on an approximating sequence of Riemannian metrics $g_n$, and is called the {\em curvature measure} of an Alexandrov surface.  As an example, consider the surface of a unit cube in $\mathbf{R}^3$. The metric on it is defined as the infimum of Euclidean lengths of all paths that lie on the surface of the cube and join two given points. As is known~\cite{Re93,Tro}, its curvature measure is $\sum(\pi/2)\delta_p$, where $\delta_p$ is the Dirac mass and the sum runs over all vertices $p$ of the cube.

Recall that a point $x\in M$ is called the {\em cusp}, if $\omega(x)=2\pi$. By the results of Reshetnyak~\cite{Re} and Huber~\cite{Huber}, any Alexandrov surface of bounded integral curvature and without cusps can be regarded as being "conformally equivalent" to a smooth Riemannian metric on a background compact surface. This means that the distance function on such a surface has the form
$$
d(x,y)=\inf_\gamma\left\{\int_0^1e^{u(\gamma(t))}\abs{\dot\gamma(t)}_gdt\right\}
$$
for some function $u$ and a smooth background Riemannian metric $g$; the infimum above is taken over smooth paths $\gamma$ joining $x$ and $y$. The conformal factor $e^u$ here can be very singular, and is an $L^{2p}$-function, where $p>1$. More precisely, the function $u$ is the difference of weakly subharmonic functions~\cite{Re,Re93},  and the set
$$
S=\{x\in M : e^u(x)=0\}
$$
has zero capacity in $M$, see~\cite[Theorem 5.9]{HK}.

Thus, an Alexandrov surface without cusps can be viewed as a surface with a "Riemannian metric" $h=e^{u}g$ on $M\backslash S$, whose distance function is precisely the original metric $d$. This "Riemannian metric"  yields the {\em Alexandrov volume measure} $d\mu_h=e^{2u}d\mathit{Vol}_g$, which is an one more intrinsic characteristic of $(M,d)$; it can be also defined via approximations by Riemannian metrics. More precisely, in~\cite{AZ} Alexandrov and Zalgaller show  that if $g_n$ is a sequence of Riemannian metrics that satisfy the hypothesis~(iii) in the definition of an Alexandrov surface, then its volume measures $\mathit{Vol}_{g_n}$ converge weakly to $\mu_h$.

Since the set $S$ has zero capacity, by conformal invariance it is straightforward to conclude that the relation
$$
\int_{M\backslash S}\abs{\nabla u}_h^2d\mu_h=\int_M\abs{\nabla u}_g^2\mathit{dVol}_g
$$
holds for any smooth function $u$. Thus, the eigenvalues $\lambda_k(g,\mu_h)$ of the Dirichlet form $\int\abs{\nabla u}^2\mathit{dVol}_g$ in $L_2(M,\mu_h)$ are indeed natural versions of Laplace eigenvalues on an Alexandrov surface without cusps.  Since $e^{2u}$ is an $L^p$-function, where $p>1$, we conclude that the Alexandrov volume measure $\mu_h$ satisfies the decay hypothesis~\eqref{cem}. In particular, the multiplicities $m_k(g,\mu_h)$ are finite and satisfy inequalities in Theorem~\ref{t4}.

\appendix
\section{Appendix: Cheng's structure theorem}
\label{stru}
The purpose of this section is to give details on Cheng's structure theorem~\cite{Cheng}, mentioned in Sect.~\ref{intro}. It is based on the following lemma.
\begin{lemma}
\label{aux}
Let $u$ be a $C^{1,1}$-smooth function defined in a neighbourhood of the origin in $\mathbf{R}^n$ that satisfies the relation
\begin{equation}
\label{ap:vo}
u(x)=P_N(x)+O(\abs{x}^{N+\delta})\qquad\text{as~ }x\to 0,
\end{equation}
where $P_N$ is a homogeneous polynomial of order $N$ such that $\abs{\nabla P_N(x)}\geqslant C\abs{x}^{N-1}$. Then there exists a neighbourhood $U$ of the origin and a Lipschitz homeomorphism $\Phi$ of it that preserves the origin and such that $u(x)=P_N(\Phi(x))$ for any $x\in U$. Moreover, if $u$ is $C^2$-smooth, then $\Phi$ is a $C^1$-diffeomorphism.
\end{lemma}
\begin{proof}[Comments on the proof]
The second term on the right-hand side can be viewed as the product $\alpha(x)\abs{x}^{N+\delta'-1}$, where $0<\delta'<\delta$ and $\alpha(x)$ is a function that is $C^1$-smooth away from the origin and behaves like $O(\abs{x}^{1+\delta-\delta'})$ as $x\to 0$. It is then straightforward to see that $\alpha$ is $C^1$-smooth in a neighbourhood of the origin, and differentiating relation~\eqref{ap:vo}, we obtain
$$
\nabla u(x)=\nabla P_N(x)+O(\abs{x}^{N+\delta'-1})\qquad\text{as~ }x\to 0.
$$
Given the last relation, if $u$ is $C^2$-smooth, the existence of the $C^1$-diffeomorphism $\Phi$ follows from the argument in the proof of~\cite[Lemma~2.4]{Cheng}. This argument also works when $u$ is $C^{1,1}$-smooth, and in this case it yields a local Lipschitz homeomorphism $\Phi$ such that $u(x)=P_N(\Phi(x))$.
\end{proof}

In dimension two any homogeneous harmonic polynomial of degree $N\geqslant 1$ satisfies the hypothesis $\abs{\nabla P_N(x)}\geqslant C\abs{x}^{N-1}$, and combining the lemma above with Prop.~\ref{vo}, we obtain the following improved version of Cheng's result.
\begin{CST}
Let $u$ be a $C^{1,1}$-smooth solution of the Schrodinger equation 
\begin{equation}
\label{schro2}
(-\Delta+V)u=0\qquad\text{on~}\Omega\subset\mathbf{R}^2,
\end{equation}
where $V\in K^{2,\delta}(\Omega)$. Then for any nodal point $p\in\mathcal N(u)$ there is a neighbourhood $U$ and a Lipschitz homeomorphism $\Phi$ of $U$ onto a neighborhood of the origin such that $u(x)=P_N(\Phi(x))$ for any $x\in U$, where $P_N$ is an approximating homogeneous harmonic polynomial at $p$. Moreover, if $u$ is $C^2$-smooth, then $\Phi$ is a $C^1$-diffeomorphism.
\end{CST}
In~\cite{Cheng} Cheng also states similar results in arbitrary dimension. However, in dimension $n>2$ there are homogeneous harmonic polynomials for which the hypothesis $\abs{\nabla P_N(x)}\geqslant C\abs{x}^{N-1}$ fails, and thus, Lemma~\ref{aux} can not be used. As is shown in~\cite[Appendix~E]{BM}, the latter hypothesis is necessary for the conclusion of Lemma~\ref{aux} to hold.


{\small
\begin{thebibliography}{99}
\addcontentsline{toc}{section}{References}

\bibitem{AS} Aizenman,~M., Simon,~B. {\em Brownian motion and Harnack inequality for Schr\"odinger operators.} Comm. Pure Appl. Math. {\bf 35} (1982), 209--273.

\bibitem{AZ} Alexandrov~A.~D., Zalgaller,~V.~A. {\em Intrinsic Geometry of Surfaces.} AMS Transl. Math. Monographs, Vol. 15, Providence, RI, 1967.

\bibitem{Bers} Bers,~L. {\em Local behavior of solutions of general linear elliptic equations.} Comm. Pure Appl. Math. {\bf 8} (1955), 473--496.

\bibitem{BM} B\'erard,~P., Meyer,~D. {\em In\'egalit\'es isop\'erim\'etriques et applications.} Ann. Sci. Ec. Norm. Sup. {\bf 15} (1982), 513--542.

\bibitem{Be} Besson,~G. {\em Sur la multiplicit\'e de la premi\'ere valeur propre des surfaces riemanniennes.}  Ann. Inst. Fourier (Grenoble) {\bf 30} (1980), 109--128.

\bibitem{SS} Canillo,~S., Sawyer,~E.~T. {\em Unique continuation for $\Delta+v$ and the C. Fefferman-Phong class.} Trans. Amer. Math. Soc. {\bf 318} (1990), 275--300. 

\bibitem{Cara} Caratheodory,~C. {\em \"Uber die Begrenzung einfach zusammenh\"engender Gebiete.} Math. Ann. {\bf 73} (1913), 323--370.


\bibitem{Chee} Cheeger,~J. {\em Spectral geometry of singular Riemannian spaces.}  J. Differential Geom. {\bf 18} (1983), 575--657.

\bibitem{Cheng} Cheng,~S.~Y. {\em Eigenfunctions and nodal sets.} Comment. Math. Helv. {\bf 51} (1976), 43--55.

\bibitem{CDV86}  Colin de Verdi\`ere,~Y. {\em Sur la multiplicit\'e de la premi\`ere valeur propre non nulle du laplacien.} Comment. Math. Helv. {\bf 61} (1986), 254--270. 

\bibitem{CDV87} Colin de Verdi\`ere,~Y. {\em Construction de laplaciens dont une partie finie du spectre est donn\'ee.} Ann. Sci. \'Ecole Norm. Sup. (4) {\bf 20} (1987), 599--615. 

\bibitem{CH} Courant,~R., Hilbert,~D. {\em Methods of Mathematical Physics.} Vol.~I, Interscience Publishers, Inc., New York, 1953. xv+561 pp. 

\bibitem{Ep} Epstein,~D. {\em Prime ends.} Proc. London Math. Soc. {\bf 42} (1981), 385--414.



\bibitem{Gib} Giblin,~P. {\em  Graphs, surfaces and homology.} Third edition. Cambridge University Press, Cambridge, 2010. 





\bibitem{HK}  Hayman,~W.~K., Kennedy,~P.~B. {\em Subharmonic functions.} Vol. I. London Mathematical Society Monographs, No. 9. Academic Press, London-New York, 1976, xvii+284 pp.

\bibitem{HoHo} Hoffmann-Ostenhof,~M., Hoffman-Ostenhof,~T. {\em Local properties of solutions of Schrodinger equations.} Comm. PDE {\bf 17} (1992), 491--522.

\bibitem{HoHoNa} Hoffmann-Ostenhof,~M., Hoffman-Ostenhof,~T., Nadirashvili, N. {\em Interior H\"older estimates for solutions of Schrodinger equations and the regularity of nodal sets.} Comm. PDE {\bf 20} (1995), 1241--1273.

\bibitem{HoNa2} Hoffmann-Ostenhof,~M., Hoffmann-Ostenhof,~T.,  Nadirashvili,~N. {\em On the multiplicity of eigenvalues of the Laplacian on surfaces.} Ann. Global Anal. Geom. {\bf 17} (1999), 43--48.

\bibitem{HoNa} Hoffmann-Ostenhof,~T., Michor,~P.~W., Nadirashvili,~N. {\em Bounds on the multiplicity of eigenvalues for fixed membranes.}  Geom. Funct. Anal. {\bf 9} (1999), 1169--1188.

\bibitem{Huber} Huber,~A. {\em Zum potentialtheoretischen Aspekt der Alexandrowschen Fl\"achentheorie.} (German)  Comment. Math. Helv. {\bf 34} (1960), 99--126.


\bibitem{KKP} Karpukhin,~M., Kokarev,~G., Polterovich,~I. {\em Multiplicity bounds for Steklov eigenvalues on Riemannian surfaces.}  Ann. Inst. Fourier, to appear; arXiv:1209.4869v2.

\bibitem{Kato} Kato,~T. {\em Perturbation theory for linear operators.} Second edition. Grundlehren der Mathematischen Wissenschaften, Band~132. Springer-Verlag, Berlin--New York, 1976. xxi+619 pp. 


\bibitem{GK} Kokarev,~G. {\em Variational aspects of Laplace eigenvalues on Riemannian surfaces.} Adv. Math. {\bf 258} (2014), 191-239.

\bibitem{GK2} Kokarev,~G. {\em Eigenvalue problems on Alexandrov surfaces of bounded integral curvature.} In preparation.

\bibitem{Ma} Maz'ja,~V.~G. {\em Sobolev spaces.} Translated from the Russian by T. O. Shaposhnikova. Springer Series in Soviet Mathematics. Springer-Verlag, Berlin, 1985. xix+486 pp.

\bibitem{Mil} Milnor,~J. {\em Dynamics in one complex variable.} Third edition. Annals of Mathematics Studies, {\bf 160}. Princeton University Press, Princeton, NJ, 2006.

\bibitem{Na} Nadirashvili,~N. {\em Multiple eigenvalues of the Laplace operator.} (Russian) Mat. Sb. (N.S.) {\bf 133(175)} (1987),  223--237, 272; translation in  Math. USSR-Sb. {\bf 61} (1988), 225--238. 


\bibitem{Re} Reshetnyak,~Y.~G. {\em Isothermal coordinates on manifolds of bounded curvature. I, II. }(Russian) Sib. Math. J. {\bf 1} (1960), 88--116, 248--276. 

\bibitem{Re93} Reshetnyak,~Y.~G. {\em Two-dimensional manifolds of bounded curvature.} Geometry, IV, 3--163, 245--250, Encyclopaedia Math. Sci., {\bf 70}, Springer, Berlin, 1993. 

\bibitem{Saw} Sawyer,~E.~T. {\em Unique continuation for Schrodinger operators in dimension three or less.} Ann. Inst. Fourier (Grenoble) {\bf 34} (1984), 189--200.

\bibitem{Se}  S\'evennec,~B. {\em Multiplicity of the second Schrodinger eigenvalue on closed surfaces.} Math. Ann. {\bf 324} (2002), 195--211. 

\bibitem{Si} Simon,~B. {\em Schrodinger semigroups.}  Bull. Amer. Math. Soc. (N.S.) {\bf 7}  (1982), 447--526.

\bibitem{Tro} Troyanov,~M. {\em Les surfaces a courbure int\'egrale born\'ee au sens d'Alexandrov.} arXiv:0906.3407.

\end{thebibliography}
}
\end{document}